\documentclass[12pt]{amsart}

\usepackage{stmaryrd}

\usepackage{amsfonts,color}
\usepackage{latexsym,amssymb}
\usepackage{amsmath}
\usepackage[utf8]{inputenc}

\newtheorem{theorem}{Theorem}[section]
\newtheorem{lemma}[theorem]{Lemma}
\newtheorem{proposition}[theorem]{Proposition}

\newtheorem{remark}[theorem]{Remark}
\newtheorem{definition}[theorem]{Definition}

\newcommand{\R}{\mathbb{R}}

\newcommand{\N}{\mathbb{N}}
\newcommand{\di}{\, \mathrm{d}}
\newcommand{\csch}{\mathrm{csch}}

\usepackage[colorlinks, linkcolor=blue, citecolor=red]{hyperref}

\catcode`@=11 \@addtoreset{equation}{section} \catcode`@=12

\usepackage[left=3cm,right=3cm,top=3cm,bottom=3cm]{geometry}

\begin{document}
	
\title[Title]{On the semilinear heat equation with the Grushin operator}

\author[G. Oliveira]{Geronimo Oliveira}
\address{Universidade Federal de Sergipe, Programa de Pós-graduação em Matemática de Matem\'atica, 49100-000 São Cristóvão-SE, Brazil}
\email{geronimofisica@academico.ufs.br}\thanks{The first author was supported by Capes/Brasil}

\author[A. Viana]{Arlúcio Viana}
\address{Universidade Federal de Sergipe, Departamento de Matem\'atica, 49100-000 São Cristóvão-SE, Brazil}
\email{arlucioviana@academico.ufs.br}\thanks{*}



\begin{abstract} 
	
	In this work, we study the heat equation with Grushin's operator. We present an expression for its heat kernel, prove its decay in $L^p$ spaces, and that it is a approximation of the identity. As a consequence, the heat semigroup associated to Grushin's operator by means of this heat kernel is estimated in Lebesgue spaces. Next, we use the results to prove the existence, uniqueness, continuous dependence and blowup alternative of mild solutions of a nonlinear Cauchy's problem associated to Grushin's operator. A global existence result is also presented.
	
\end{abstract}

\maketitle

\section{Introduction}

\hspace{12pt} The celebrated heat equation is given by
\begin{equation}\label{heat}
	u_t = \Delta u , \ \ \ (x,t)\in \Omega\times(0,T) \subset \R^N\times(0, \infty),
\end{equation} 
where $\Delta$ denotes the Laplacian in the variable $x$ and $\Omega$ is a domain. It is used to model various physical, chemical and biological phenomena. It is well-known that Fourier law is used to deduce \eqref{heat}, that is, the flux is proportional to $-\nabla u$. Several variants of this equation were introduced in order to obtain more realistic models. A way to get variants that may approach certain diffusion problems is allowing other possibilities to the heat flux. Thus, other operators may replace the Laplacian.

Sub-elliptic operators are of great scientific interest since the work by Chow \cite{Chow-39}. They are naturally connected with diffusion problems where the flux does not flow as in the direction of the standard gradient. In these cases, a new heat equation appears and the new heat kernel needs to be found and studied. we cite some works that treat the heat kernel of those type of operators \cite{Gaveau-77, Arous-89, Chang-Li-15, Ba-Furutani-15}. We can find a good survey of techniques to find the heat kernel of sub-elliptic operators in Callin's \textit{et al} monograph \cite{Calin-11}. Other works on sub-elliptic operators are \cite{Ar-B-08,Calin-11, Chang-Li-12,Goldstein-Kogoj, Kogoj-Lan-12,Wu-15}.


The Grushin operator appeared in Grushin's paper \cite{Grushin} and is defined by 
\begin{equation}\label{Grushiop}
	\Delta_{\mathcal{G}}=\dfrac{1}{2}\left(\partial_{x}^2+x^2\partial_{y}^2\right).
\end{equation}

A generalization of \eqref{Grushiop} is given by
\begin{equation}\label{Grushin2}
	\Delta_\gamma u (z) := \Delta_x u(z) + |x|^{2\gamma}\Delta_y u(z),
\end{equation}
with $\gamma>0$, $z=(x,y) \in \mathbb{R}^m\times \mathbb{R}^k$, $m+k = N\geq3$, $\Delta_x$ and $\Delta_y$ are the Laplace operators in the variables $x$ and $y$. In \cite{DAmb}, for example, the author proves nice Hardy Inequalities related to generalized Grushin operator \eqref{Grushin2}. When $k=\gamma=1$,  the kernel is similar to the heat kernel of \eqref{Grushiop}, replacing $x$ with $|x|$. See \eqref{HK} and \cite{Chang-Li-15}. The heat kernel for \eqref{Grushin2}, when $k>1$ and $\gamma=1$, appeared recently in \cite{Garofalo-Trallli-22}.

In parallel of the interest for the solutions of the linear equations mentioned above, there is a huge attention on the following Cauchy problem and its variants:
\begin{eqnarray}\label{sheat}
	u_t(x,t) = \Delta u(x,t) + |u(x,t)|^{\rho-1}u(x,t),\  \mbox{in}\ (0,\infty)\times \R^N ,\\
	u(x,0)=u_0(x),\ \mbox{in}\ \R^N . \label{sheat0}
\end{eqnarray} 
Here, $u_0(x)$ is the initial distribution of the density $u$ and $\rho>1$. Indeed, Fujita, in his seminal paper \cite{Fujita}, proved the following
If $u_0\in C_0(\R^N)$ be nonnegative and nonzero, then
\begin{enumerate}
	\item If $1<\rho<1+\frac{2}{ N}$, there exists no positive global solution of \eqref{sheat}--\eqref{sheat0}.
	\item If $\rho> 1+\frac{2}{ N}$, there exists $u_0\in L^{\frac{ N}{2}(\rho-1)}(\R^N)$ such that there exists a global positive solution of \eqref{sheat}--\eqref{sheat0}.
\end{enumerate}
The critical case $\rho = 1+\frac{2}{N}$ was resolved by Weissler \cite{Weissler}. Since then, a huge amount of work dedicate to treat \eqref{sheat} and its generalizations. The study of partial differential equations in spaces with low regularity, such as Lebesgue, provides us weaker notions for solutions of the studied equations, allowing its solvability, and modeling natural situations where the initial data are not regular. Heat equation \eqref{heat} was already treated in Lebesgue spaces \cite{BrezisCaz,Weissler}. 

In this work, we want to give sufficient conditions to the local existence of $L^q$-mild solutions for the Cauchy problem
\begin{eqnarray}\label{sheatG}
	u_t = \Delta_\mathcal{G} u + |u|^{\rho-1}u,\  \mbox{in}\ (0,\infty)\times \R^{N+k} ,\\
	u(0)=u_0,\ \mbox{in}\ \R^{N+k} , \label{sheatG0}
\end{eqnarray} 
where $\Delta_\mathcal{G} = \dfrac{1}{2}\left(\Delta_{x}+|x|^2\Delta_{y}\right)$ and $\Delta_x, \Delta_y$ denote the classical Laplacian in the variables $x\in\R^N$ and $y\in\R^k$, respectively. More precisely, we give sufficient conditions to the existence of local and global mild solutions of \eqref{sheatG}--\eqref{sheatG0}, for initial conditions in $L^p(\R^{N+k})$. We prove that the local solutions can be extended to a maximal interval $(0,T_{max})$ for which either $T_{max}<\infty$ or the norm of the solutions blows-up in finite time.

A problem similar to \eqref{sheatG}--\eqref{sheatG0}, for $N=k=1$, was studied in \cite{Lv-Wei-19}, where the authors use the kernel
\begin{equation}\label{Lv}
	K(x,y;t)=\dfrac{\sqrt{2}}{2\pi t^{3/2}}e^{-\frac{x^2}{2t}-\frac{y^2}{t^2}} ,
\end{equation}
which is the heat kernel for for the operator $\frac{1}{2}(\partial_{x}^2u+t\partial_{y}^2u)$. They provided sufficient conditions for the global existence and for blow-up of solutions to an integral equation involving the heat kernel \eqref{Lv}.

A key ingredient in the analysis of the Cauchy problem \eqref{sheat}--\eqref{sheat0} is the estimates for the solution of the linear part of \eqref{sheatG}. Consequently, it is fundamental to prove $L^p$ estimates for the heat kernel associated with the equation. In this work, we give a proof of such estimates by analyzing the representation of the heat kernel. The fact that this kernel is given in terms of the partial Fourier transform, and not explicitly, brings substantial difficulties in the analysis. By following the Geometric Method in \cite{Calin-11}, we find that 
\begin{equation}\label{kernelpq}
	L(x, x_0,t)=\left(
	\dfrac{a}{2\pi\sinh(at)}\right)^{N/2}e^{-\frac{a}{2\sinh(at)} \left(  (|x|^2 + |x_0|^2) \cosh(at)  -2x\cdot x_0 \right)},\ t>0,
\end{equation}
is the heat kernel of the heat equation with quadratic potential 
\begin{equation}\label{heatpq}
	u_t =	\frac{1}{2} \Delta - \frac{1}{2}a^2|x|^2 ,
\end{equation}
with the endpoint of the path for the classical action being  $x= x_0$.

Then, if we apply the partial Fourier transform on the variable $y$ in	$u_t = \Delta_\mathcal{G} u$, we get \eqref{heatpq} with $a=|\xi|$, where $\xi$ is the Fourier variable. Then, the inverse partial Fourier transform of \eqref{kernelpq} gives
\begin{equation}\label{HK}
	K(x, x_0,y;t) = \frac{1}{(2\pi)^{\frac{N+2k}{2}}} \int_{\R^k} \left(\frac{|\xi|}{\sinh(|\xi| t)}\right)^{\frac{N}{2}} e^{i\xi  \cdot y- \frac{|\xi|}{2}\left(  (|x|^2 + |x_0|^2) \coth(|\xi| t)  -2x\cdot x_0 \csch(|\xi|t) \right)} \di\xi, 
\end{equation}
for $(x, x_0,y) \in \R^{2N+k},\ t>0.$ This expression also appears in Theorem 3.4 of the paper by Garofalo and Tralli \cite{Garofalo-Trallli-22}, with slight differences due to constant choices in the definition of the Grushin operator. They proved that \eqref{HK} is a solutions of 
\begin{eqnarray*}\label{heatG}
	u_t = \Delta_\mathcal{G} u , \  \mbox{in}\ (0,\infty)\times \R^{N+k} ,\\
	u(0)=\delta(x)\otimes\delta(y),\ \mbox{in}\ \R^{N+k} .
\end{eqnarray*} 

We also refer to the recent Stempak's paper \cite{Stempak-25} for a closed formula of the heat kernel obtained by means of a scalar transform which is a mixture of the partial Fourier transform and a transform based on the scaled Hermite functions.

\section{The heat kernel and the heat semigroup in $L^p$}

In the introduction, we saw that the heat kernel of the equation $u_t = \Delta_\mathcal{G} u$ is given by \eqref{HK}. One may perform straightforward although long computations to prove that \eqref{HK} is $C^\infty$ and is actually the kernel of the equation, that is, 
\begin{enumerate}
	\item $K_t = \Delta_\mathcal{G} K$;
	\item $\int_{\R^{N+k}} K(x,0,y;t) d(x,y) =1$ ;
	\item $\lim_{t\rightarrow0^+}\int_{\R^{N+k}} K(x,w,y-z;t) \varphi(w,z) d(w,z) = \varphi(x,y)$ .
\end{enumerate}

For completeness, we include the following well-known result.

\begin{proposition}\normalfont\label{Teoremanucleo}
	The heat kernel of $\Delta_{\mathcal{G}}=\frac{1}{2}(\Delta_{x}+|x|^2\Delta_{y})$ is given by \eqref{HK}.
\end{proposition}

\begin{proof}
	In fact, for any $x_0\in\R^N$, denote
	\begin{equation}\label{fkernel}
		f(x, x_0,\xi;t) = \left(\frac{|\xi|}{\sinh(|\xi| t)}\right)^{\frac{N}{2}} e^{-\frac{|\xi|}{2}\left(  (|x|^2 + |x_0|^2) \coth(|\xi| t)  -2x\cdot x_0 \csch(|\xi|t) \right)}.
	\end{equation}
	Then, straightforward computations lead to
	\begin{align*}
		\partial_{x_i}f(x, x_0,\xi;t) & 
		= - f(x, x_0,\xi;t)  |\xi| (x_i\coth(|\xi| t)  -x_{0_i} \csch(|\xi|t) ); \\[.1cm]
		\partial_{x_i}^2f(x, x_0,\xi;t) 
		&=f(x, x_0,\xi;t)\left(  (x_i|\xi|\coth(|\xi| t) - |\xi|x_{0_i} \csch(|\xi|t) )^2-|\xi|\coth(|\xi| t)\right); \\[.1cm]
		\partial_{t}f(x, x_0,\xi;t)
		&=\dfrac{|\xi|^2\csch^2(|\xi| t) }{2}f(x, x_0,\xi;t) ((|x|^2 +  |x_0|^2|)\\
		& \ \   - 2x\cdot x_0 \coth(|\xi| t)\csch(|\xi| t)  -N \coth(|\xi| t)).
	\end{align*}
	Then, we can derivative under the sign of the integral
	to get $\partial_tK-\Delta_{\mathcal{G}}K=0$. 
	
	Now, from Fubini's theorem, we rewrite
	\begin{align*}
		&\int_{\R^{N+k}}K(x, x_0,y;t) \di(x,y) \\
		= & \int_{\R^k}\int_{\R^k}\left(\dfrac{1}{(2\pi)^{\frac{N+2k}{2}}} \int_{\R^N} \left(\frac{|\xi|}{\sinh(|\xi| t)}\right)^{\frac{N}{2}} e^{i \xi \cdot y -\frac{|\xi|}{2}\left(  (|x|^2 + |x_0|^2) \coth(|\xi| t)  -2x\cdot x_0 \csch(|\xi|t) \right)} \di x \di y\right) \di\xi .
	\end{align*}
	The identity
	\begin{equation*}
		\int_{\R^N} e^{-a |x|^2+ b\cdot x} \di x=\left(\frac{\pi}{a}\right)^{\frac{N}{2}} e^{\frac{|b|^2}{4 a}}, \ \forall \ a >0,  \ b \in \mathbb{R^N},
	\end{equation*}
	gives then
	\begin{align*}
		&\int_{\R^{N+k}}K(x, x_0,y;t) \di(x,y) \\
		= & \dfrac{1}{(2\pi)^{k}} \int_{\R^k} \int_{\R^k} \left(\frac{|\xi|}{\sinh(|\xi| t)}\right)^{\frac{N}{2}}  e^{i\xi \cdot y} \left(\frac{1}{|\xi|\coth(|\xi| t)}\right)^{\frac{N}{2}} e^{ \frac{ |x_0|^2}{2}|\xi|\left( \csch(|\xi| t)\mathrm{sech}(|\xi| t) -  \coth(|\xi| t) \right)} \di y  \di\xi \\[.1cm]
		= & \dfrac{1}{(2\pi)^k} \int_{\R^k}\int_{\R^k} \left(\frac{1}{\cosh(|\xi| t)}\right)^{\frac{N}{2}} e^{ \frac{ |x_0|^2}{2}|\xi|\left( \csch(|\xi| t)\mathrm{sech}(|\xi| t) -  \coth(|\xi| t) \right)} e^{i \xi \cdot y} \di y \di\xi.
	\end{align*}
	Define
	\begin{equation*}
		I(s):=\dfrac{1}{(2\pi)^k} \int_{\R^k}\int_{\R^k} \left(\frac{1}{\cosh(|\xi| t)}\right)^{\frac{N}{2}} e^{\frac{ |x_0|^2}{2} |\xi|\left(\csch(|\xi| t)\mathrm{sech}(|\xi| t) -  \coth(|\xi| t) \right)} e^{i \xi \cdot y} e^{-s|y|^2} dy \di \xi .
	\end{equation*}
	Hence, using $\xi = 2\sqrt{s}\eta$, we have
	\begin{align*}
		I(s)
		&= \frac{1}{(4\pi s)^{\frac{k}{2}}} \int_{\R^k} \left(\frac{1}{\cosh(|\xi| t)}\right)^{\frac{N}{2}} e^{ \frac{ |x_0|^2}{2} |\xi|\left( \csch(|\xi| t)\mathrm{sech}(|\xi| t) -  \coth(|\xi| t) \right)} e^{-\frac{|\xi|^2}{4s}} \di \xi  \\
		&= \frac{1}{\pi^{\frac{k}{2}}} \int_{\R^k} \left(\frac{1}{\cosh(|\eta| 2\sqrt{s} t)}\right)^{\frac{N}{2}} e^{ \frac{ |x_0|^2}{2} |2\sqrt{s}\eta|\left( \csch(|2\sqrt{s}\eta| t)\mathrm{sech}(|2\sqrt{s}\eta| t) -  \coth(|2\sqrt{s}\eta| t) \right)} e^{-|\eta|^2}d\eta.
	\end{align*}
	
	Taking the limit as $s\rightarrow 0^+$, we get 
	\begin{equation}\label{mass}
		\int_{\R^{N+k}}K(x, x_0,y;t)d(x,y) =  1 ,
	\end{equation}	
	for all $t>0$ and $x_0\in\R^N$. We used that the function $h:\R_+ \to\R$ given by
	$$h_t(b) =  b\left( \csch(b t)\mathrm{sech}(b t) -  \coth(b t) \right)$$ 
	satisfies $h_t(0) = 0$, which can be proved by direct calculations.
\end{proof}

\begin{remark}\label{remarkHK}
	Let us underline some facts on the heat kernel \eqref{HK}.
	\begin{enumerate}
		\item The idea in the beginning of latter proof can be used to prove that $K\in C^\infty$. 
		\item The kernel is symmetric in the following sense: $K(x, x_0,y;t) = K( x_0,x,y;t)$ and $K(x, x_0,-y;t) = K(x, x_0,y;t)$. Furthermore, the change of variables $w=\xi t$ in \eqref{HK} leads to
		\begin{equation}\label{symmetry}
			K(x, x_0,y;t)= t^{-\frac{N+2k}{2}} K(t^{-\frac{1}{2}}x, t^{-\frac{1}{2}}x_0, t^{-1}y;1).
		\end{equation} 
		\item As a consequence of \eqref{symmetry}, $K$ can be regarded as an approximation of identity. See Lemma \ref{approxid} below.
		\item The positivity of $K$ is a delicate subject. Fortunately, it was proved in \cite{Chang-Gaveau-09,Chang-Greiner-13} that \eqref{HK} is positive, for $N=k=1$. For higher dimensions, see Theorem 2.4 of Biagi and Bramanti's paper \cite{Bi-Br-23}. 
		\item The positivity of the kernel and with Proposition \ref{Teoremanucleo} allows to conclude that, for any $x_0\in\R^N$ and $t>0$,
		\begin{equation}\label{K1}
			\|K(\cdot,x_0,\cdot;t)\|_{L^1(\R^{N+k})} =1.
		\end{equation}
	\end{enumerate}	
\end{remark}

The next result provides the estimates of the heat kernel in $L^p$, which are key estimates to treat \eqref{sheatG}--\eqref{sheatG0} in Lebesgue spaces.

\begin{theorem}\label{kernelpropertys}
	Let $K$ be the heat kernel \eqref{HK}. If $1\leq q \leq \infty$, then $$\|K(\cdot,x_0,\cdot;t)\|_{L^q(\R^{N+k})}\leq Ct^{\frac{N+2k}{2}\left(1-\frac{1}{q}\right)}  ,$$
	for any $x_0\in\R^N$ and $t>0$.
\end{theorem}

\begin{proof}
	Notice that $K(\cdot,x_0,y;\cdot)=\mathcal{F}^{-1}[f](\cdot,x_0,\xi;\cdot)=\mathcal{F}[f](\cdot,x_0,-\xi;\cdot)=\mathcal{F}[f](\cdot,x_0,\xi;\cdot)$, where $f$ is defined in \eqref{fkernel}.
	We split the proof into three cases: i) $2\leq q<\infty$, ii) $1\leq q<2$, iii) $q=\infty$.

	We can apply Hausdorff-Young inequality, with $1< p\leq 2$ and $p^{-1}+q^{-1}=1$, to estimate
	\begin{equation*}
		\|K(x,x_0,\cdot;t)\|^q_{L^q(\R^k)}\leq \|f(x,x_0,\cdot,t)\|^q_{L^p(\R^k)}.
	\end{equation*}
	Accordingly,
	\begin{align}
		\|K(\cdot,x_0,\cdot;t)\|_{L^q(\R^{N+k})}&\leq 
		\left(\left(\int_{\R^N}\left(\int_{\R^k}|f(x,x_0,\xi;t)|^p\di \xi \right)^{q/p}\di x \right)^{p/q}\right)^{1/p}.\label{i2}
	\end{align}
	Since $\frac{q}{p}=\frac{1}{p-1}\geq 1$ and $|f|\geq 0$, Minkowski inequality for integrals yields
	\begin{equation*}
		\left(\int_{\R^N}\left(\int_{\R^k}|f(x,x_0,\xi;t)|^p\di \xi \right)^{q/p}\di x \right)^{p/q}\leq \int_{\R^k}\left(\int_{\R^N}|f(x,x_0,\xi;t)|^{p\cdot\frac{q}{p}}\di x \right)^{p/q}\di \xi .
	\end{equation*}
	Recalling the expression for the function $f$, we write
	\begin{align*}
		&\int_{\R^N}\left(\int_{\R^k}|f(x,x_0,\xi;t)|^q\di x \right)^{p/q}\di \xi  \\
		= & \int_{\R^k}\left(\int_{\R^N} \left(\frac{|\xi|}{\sinh(|\xi| t)}\right)^{\frac{Nq}{2}} e^{-\frac{q|\xi|}{2}\left(  (|x|^2 + |x_0|^2) \coth(|\xi| t)  -2x\cdot x_0 \csch(|\xi|t) \right)} \di x \right)^{p/q} \di\xi \\[.1cm]
		= & \int_{\R^k} \left(\frac{|\xi|}{\sinh(|\xi| t)}\right)^{\frac{Np}{2}} \left(
		\dfrac{2\pi\tanh(|\xi| t)}{|\xi| q}\right)^{\frac{Np}{2q}} e^{ \frac{ p|x_0|^2}{2}\xi|\left( \csch(|\xi| t)\mathrm{sech}(|\xi| t) -  \coth(|\xi| t) \right)} \di \xi  \\[.1cm]
		\leq & \left(\frac{2\pi}{q}\right)^\frac{Np}{2q} t^{-\frac{N}{2}} \int_{\R^k}\cosh^{-\frac{Np}{2q}}(|\xi| t) e^{ \frac{ p|x_0|^2}{2}|\xi|\left( \csch(|\xi| t)\mathrm{sech}(|\xi| t) -  \coth(|\xi| t) \right)} \di \xi .
	\end{align*}
	Recall that 
	$$|\xi|\left(\csch(|\xi| t)\mathrm{sech}(|\xi| t) -  \coth(|\xi| t) \right) \leq 0 .$$ Then, it is sufficient estimating the $\cosh$ and using polar coordinates, as follows
	\begin{equation*}
		\int_{\R^k}\cosh^{-\frac{Np}{2q}}(|\xi| t)\di \xi  
		\leq 
		2^{1+Np/2q} \int_{0}^{\infty}
		e^{-\frac{r t Np}{2q}} r^{k-1} \di \xi  = (k-1)! \left(\frac{ t Np}{2q}\right)^k  = c_{p,q,N,k} t^{-k} .
	\end{equation*}
	So,
	\begin{align}\label{i3}
		\int_{\R^k}\left(\int_{\R^N}|f(x,x_0,\xi;t)|^q\di x \right)^{p/q}\di \xi \leq  C'(p,q,N,k)t^{-\frac{N+2k}{2}}  .
	\end{align}

	From \eqref{i2} and \eqref{i3}, we conclude that
	\begin{align*}
		\|K(\cdot,x_0,\cdot;t)\|_{L^q(\R^{N+k})}\leq  C(p,q,N,k)t^{-\frac{N+2k}{2}\left(1-\frac{1}{q}\right)} ,
	\end{align*}
	with $ 2 \leq q < \infty.$ This concludes i).

	Notice that the estimate is already proved in i), for $q=2$. From Remark \ref{remarkHK} (6), we get the case $q=1$. Indeed, we have $	\|K(\cdot,x_0,\cdot;t)\|_{L^1(\R^{N+k})}=1$ and $	\|K(\cdot,x_0,\cdot;t)\|_{L^2(\R^{N+k})}\leq C t^{-\frac{N+2k}{4}} $. Then, for $1< q<2$, we can interpolate Lebesgue spaces $L^1$ and $L^2$, with $\lambda = \frac{2}{q}-1 \in (0,1)$:
	\begin{align*}
		\|K(\cdot,x_0,\cdot;t)\|_{L^q(\R^{N+k})}&\leq \|K(\cdot,x_0,\cdot;t)\|_{L^1(\R^{N+k})}^{\frac{2}{q}-1}\|K(\cdot,x_0,\cdot;t)\|_{L^2(\R^{N+k})}^{2-\frac{2}{q}} \\ 
		&\leq Ct^{-\frac{N+2k}{2}\left(1-\frac{1}{q}\right)} .
	\end{align*}
	It proves ii).

	For $q=\infty$, case iii), it is sufficient to make the change of variable $\tau=\xi t$ in \eqref{HK} and use that the function $s\mapsto \frac{s}\sinh(s)$ is integrable.
\end{proof}


In the following lemma, we will denote $K_t(x,w,y) : = K(x,w,y,t)$. Then, we will prove that the kernel $K_t$ is an approximation of the identity in the following sense.
\begin{lemma}\label{approxid}
	\begin{enumerate}
		\item If $\varphi$ is bounded and uniformly continuous in $\R^{N+k}$, then $$\lim_{t\rightarrow0^+} \int_{\R^{N+k}} K_t(x,w,y-z) \varphi(w,z) \di(w,z) \rightarrow \varphi(x,y) ,$$
		uniformly.
		\item If $\varphi\in L^p(\R^{N+k})$, $1\leq p<\infty$, then 
		$$\left\|\lim_{t\rightarrow0^+} \int_{\R^{N+k}} K_t(\cdot,w,\cdot-z) \varphi(w,z) \di(w,z) - \varphi \right\|_{L^p} =0 .$$
		\item If $\varphi\in L^\infty(\R^{N+k})$, then 
		$$\lim_{t\rightarrow0^+} \int_{\R^{N+k}}  K_t(x,w,y-z) \varphi(w,z) \di(w,z) \rightarrow \varphi(x,y) ,$$
		uniformly in compact sets.
	\end{enumerate}
\end{lemma}

\begin{proof}
	From \eqref{symmetry}, we have $K_t(x,x-w,z) = K_1(t^{-\frac{1}{2}}x, t^{-\frac{1}{2}}(x-w), t^{-1}z)$. Also, from the symmetry of $K$ and \eqref{K1}, we have 
	$$\|K_t(x,x-\cdot,\cdot)\|_{L^1(\R^{N+k})}  = 1 \ \mbox{and} \  \|K_1(t^{-\frac{1}{2}}x,\cdot,\cdot)\|_{L^1(\R^{N+k})} =1.$$
	Therefore, we can write
	\begin{align*}
		&\int_{\R^{N+k}} K_t(x,w,y-z) \varphi(w,z) \di(w,z) - \varphi(x,y) \\
		= & \int_{\R^{N+k}} K_t(x,x-w,y-z) \varphi(w,z) \di(w,z) - \int_{\R^{N+k}} K_t(x,x-w,y-z)\varphi(x,y) \di(w,z)\\
		= & \int_{\R^{N+k}} K_t(x,x-w,y-z) \left[ \varphi(w,z) - \varphi(x,y) \right]  \di(w,z)  \\
		= & \int_{\R^{N+k}} K_1(t^{-\frac{1}{2}}x,t^{-\frac{1}{2}}x-w,z) \left[ \varphi(x-t^{\frac{1}{2}}w,y - tz) - \varphi(x,y) \right]  \di(w,z)  \\
	\end{align*}
	We notice that small modifications in the proof of Th. 8.5 in \cite{Folland-book} shows that 
	$$\| \varphi(\cdot-t^{\frac{1}{2}}w,\cdot - tz) - \varphi \|_{L^p(\R^{N+k},\di(x,y))} \rightarrow 0,\ \ \mbox{as}\ \ t\rightarrow0.$$
	From now on, thanks to Theorem \ref{kernelpropertys}, the same proof as in the convolution case (see e.g. \cite[Th. 8.14]{Folland-book}) will work well as soon as $K_1(t^{-\frac{1}{2}}x,t^{-\frac{1}{2}}x-w,z)$ is dominated by some integrable function on $(w,z)\in \R^{N+k}$, for all $t>0$. Therefore, the rest of this proof will be dedicated to achieve this domination. We recall that $K_1$ is positive in $\R^{2N+k}$. Moreover, notice that 
	$$|t^{-\frac{1}{2}}x|^2 + |t^{-\frac{1}{2}}x -x|^2   =  2t^{-1}|x|^2 + |w|^2 - 2t^{-\frac{1}{2}} \ x\cdot w .$$
	Then
	\begin{align}
		&\left(|t^{-\frac{1}{2}}x|^2 + |t^{-\frac{1}{2}}x -x|^2\right)\coth(|\xi|) - 2(t^{-\frac{1}{2}}  x) \cdot (t^{-\frac{1}{2}}w) \csch(|\xi|) \nonumber\\
		= & \left( 2t^{-1}|x|^2  - 2t^{-\frac{1}{2}} \ x\cdot w\right)\left(\coth(|\xi|) - \csch(|\xi|)\right) + |w|^2 \coth(|\xi|) . \label{bounde}
	\end{align}
	Since
	$$\lim_{t\rightarrow0^+} \left( 2t^{-1}|x|^2  - 2t^{-\frac{1}{2}} \ x\cdot w\right) = \infty,$$
	we already can see that 
	\begin{eqnarray*}
		\lim_{t\rightarrow0^+}\left(\frac{|\xi|}{\sinh(|\xi|)}\right)^{\frac{N}{2}}  e^{i\xi \cdot y}e^{-\left(|t^{-\frac{1}{2}}x|^2 + |t^{-\frac{1}{2}}x -x|^2\right)\coth(|\xi|) - 2(t^{-\frac{1}{2}}  x) \cdot (t^{-\frac{1}{2}}w) \csch(|\xi|)}  =0,
	\end{eqnarray*}
	whence
	\begin{equation}\label{K1-0}
		\lim_{t\rightarrow0^+} K_1(t^{-\frac{1}{2}}x,t^{-\frac{1}{2}}x-w,z) =0.
	\end{equation}
	Nevertheless, we would like this convergence to be uniform also on $w$, since it is already uniform on $z$. For this, notice that the minimum value for the function $w \mapsto - 2t^{-\frac{1}{2}} |x||w|\left(\coth(|\xi|) - \csch(|\xi|)\right) + |w|^2 \coth(|\xi|) $ is 
	$$-\frac{4t^{-1}\left(\coth(|\xi|) - \csch(|\xi|)\right)^2 |x|^2}{4\coth(\xi)} . $$
	From this and \eqref{bounde},
	\begin{align}
		& \left(|t^{-\frac{1}{2}}x|^2 + |t^{-\frac{1}{2}}x -x|^2\right)\coth(|\xi|) - 2(t^{-\frac{1}{2}}  x) \cdot (t^{-\frac{1}{2}}w) \csch(|\xi|) \nonumber \\
		& \geq 2t^{-1}|x|^2\left(\coth(|\xi|) - \csch(|\xi|)\right) -\frac{t^{-1}\left(\coth(|\xi|) - \csch(|\xi|)\right)^2 |x|^2}{\coth(\xi)} \nonumber \\
		& = t^{-1}|x|^2 \left(\coth(|\xi|) - \csch(|\xi|)\right) \left[ 2 - \frac{\left(\coth(|\xi|) - \csch(|\xi|)\right)}{\coth(|\xi|)}\right] \nonumber \\
		& > t^{-1}|x|^2 \left(\coth(|\xi|) - \csch(|\xi|)\right) , \nonumber
	\end{align}
	since 
	$$0 \leq \frac{\left(\coth(|\xi|) - \csch(|\xi|)\right)}{\coth(|\xi|)} < 1 . $$
	Therefore,
	\begin{align*}
		& e^{- \frac{|\xi|}{2}\left(  (|t^{-\frac{1}{2}}x|^2 + |t^{-\frac{1}{2}}x-w|^2) \coth(|\xi| t)  -2(t^{-\frac{1}{2}}x)\cdot (t^{-\frac{1}{2}}x-w) \csch(|\xi|t) \right)} \\ 
		\leq & e^{- \frac{|\xi|}{2}\left(t^{-1}|x|^2 \left(\coth(|\xi|) - \csch(|\xi|)\right)\right) } \\
		\rightarrow & 0, \ \mbox{as} \ \ t\rightarrow0^+,
	\end{align*}
	uniformly on $w\in\R^N$. Hence, \eqref{K1-0} holds uniformly for $(w,z) \in \R^{N+k}$. This means that, for each $x\in\R^N$, there exists $t_0>0$ such that the above computations and the nonnegative of the kernel ensure that
	$$0\leq K_1(t^{-\frac{1}{2}}x,t^{-\frac{1}{2}}x-w,z) \leq K_1(0,w,z) , $$
	for all $(w,z) \in \R^{N+k}$ and for all $0<t<t_0$. We finished the proof.
	
\end{proof}

Next, we consider the following Cauchy problem
\begin{eqnarray}\label{cauchy}
	\partial_{t}u-\Delta_{\mathcal{G}}u=0, &  (x,y)\in \R^{N+k}, \ t>0 ,\\
	\hspace*{-1cm}u(x,y,0)=u_0(x,y), &  (x,y)\in \R^{N+k} \label{cauchy0}.
\end{eqnarray}
By direct computations, we can show that 
\begin{equation}\label{sol}
	u(x,y,t)=\int_{\R^{N+k}} K(x,w,y-z,t) u_0(w,z) \di(w,z) , \ (w,z) \in \R^{N+k}, \ t>0,
\end{equation}
satisfies the equation
$$u_{t}(x,y,t)-\Delta_{\mathcal{G}} u(x,y,t)=0.$$ 
In fact, one may check, by the proof of Proposition \ref{Teoremanucleo}, that the spatial derivatives of $K$ are bounded by 
$$\frac{1}{(2\pi)^{\frac{N+2k}{2}}} \int_{\R^k} \left(\frac{|\xi|}{\sinh(|\xi| t)}\right)^{\frac{N}{2}} \mathrm{d}\xi ,$$ 
the derivatives can pass under the sign of the integral, so that the solution \eqref{sol} is smooth. Alternatively, this fact can be also proved by the Fourier transform and the solution of the harmonic oscillator as in \cite{Garofalo-Trallli-22}, where the authors also showed that, for $u_0\in \mathcal{S}(\R^{N+k})$, we have
$$\lim\limits_{(x,y,t)\rightarrow  (x_0,y_0,0)} u(x,y,t)=u_0(x_0,y_0), \ \text{a.e. in} \  \R^{N+k}.$$
Namely, \eqref{sol} is a solution of the Cauchy problem \eqref{cauchy}--\eqref{cauchy0}. The uniqueness follows from the fact that $u(x,y,t)\rightarrow 0$ as $| (x,y)|\rightarrow \infty$, and the following standard computation: the function $u=v-w$ is a solution of \eqref{cauchy} with $u_0=0$. Then, we multiply that equation by $u$ and integrate over $\R^{N+k}$ yielding
\begin{equation*}
	\dfrac{1}{2}\dfrac{d}{dt}\int_{\R^{N+k}}u^2 d(x,y) = -\int_{\R^{N+k}} \left(|\nabla_x u|^2 + |x|^2 |\nabla_{y} u|^2 \right) d(x,y) \leq 0.
\end{equation*}
Hence, $u\equiv 0$, which proves the uniqueness.

\begin{remark}
	Although $K(x,0,y;t)$ is a solution of $u_{t}(x,y,t)-\Delta_{\mathcal{G}} u(x,y,t)=0$, the function 
	$$v(x,y,t)=\int_{\R^{N+k}} K(x-w,0,y-z,t) u_0(w,z) \di(w,z)$$
	is not so, unlike $u$ defined by \eqref{sol}. This fact is related to the fact that the Grushin operator is not translation invariant. We thank Prof. Giulio Tralli for pointing it out to us.
\end{remark}

Therefore, we may define the semigroup associated with Grushin operator. In fact, define $\left(S_{\mathcal{G}}(t)\right)_{t\geq 0}$ by
\begin{equation}\label{semigroup2}
	S_\mathcal{G}(t)\varphi(x,y) = \int_{\R^{N+k}} K(x,w,y-z,t) \varphi(w,z) \di(w,z),
\end{equation}
whenever the integral on the right-hand side makes sense. We will show that $\left(S_{\mathcal{G}}(t)\right)_{t\geq 0}$ is indeed a strongly continuous semigroup in $L^p(\R^{N+k})$, for $1\leq p\leq \infty$. Moreover, we will have $L^p-L^r$ estimates.

\begin{theorem}\label{teo3.6}
	For all $1\leq p \leq \infty$, $S_\mathcal{G}(t):L^p(\R^{N+k})\rightarrow L^p(\R^{N+k})$ is a semigroup. If $1\leq p \leq r\leq\infty$, then
	\begin{equation}\label{semigroup}
		\|S_\mathcal{G}(t)\varphi\|_{L^r(\R^{N+k})}\leq C\|\varphi\|_{L^p(\R^{N+k})}t^{-\frac{N+2k}{2}\left(\frac{1}{p}-\frac{1}{r}\right)};
	\end{equation}
	Furthermore, for $1 \leq p \leq r \leq \infty$, for all $t_0>0$, it is strongly continuous, that is, for $\varphi \in L^r(\R^{N+k})$, it holds
	\begin{equation}\label{continuidadeforte} 
		\|S_\mathcal{G}(t)\varphi-S_\mathcal{G}(t_0)\varphi\|_{L^p}\rightarrow 0,
	\end{equation}
	as $t\rightarrow t_0$. When $p=r<\infty$, then we may take $t_0=0$.
\end{theorem}

\begin{proof}
	Given $\varphi\in L^p(\R^{N+k})$, we can see that $u(x, y,t)=S_{\mathcal{G}}(t)\varphi(x, y)$ is the unique (classical) solution of \eqref{cauchy}--\eqref{cauchy0}, with initial datum $\varphi$. Plus, Young inequality for integral operators in $L^p$ (see e.g. \cite{Sogge}) and Theorem \ref{kernelpropertys} give $u \in L^p(\R^{N+k})$. Next, define $v(x, y,s)=u(x, y,t+s)$. So, $v$ is the solution of
	\begin{equation*}
		\left\{\begin{array}{cc}
			v_s(x, y,s) =\Delta_\mathcal{G}v(x, y,s), \\
			v(x, y,0)=u(x, y,t).
		\end{array}\right.
	\end{equation*}
	By uniqueness, $v(x, y,s)=S_{\mathcal{G}}(s)u(x, y,t)$. Then,
	\begin{equation*}
		S_{\mathcal{G}}(t+s) \varphi(x, y,t) = u(x, y,t+s) = v(x, y,s) = S_{\mathcal{G}}(s)u(x, y,t) = S_{\mathcal{G}}(s)S_{\mathcal{G}}(t)\varphi(x, y,t).
	\end{equation*}
	
	Next, consider $1\leq p,q,r\leq\infty$, with $p^{-1}+q^{-1}=r^{-1}+1$, and $\varphi \in L^p$. From Young inequality for integral operators in $L^p$ (see e.g. \cite{Sogge}) and Theorem \ref{kernelpropertys}, we have
	\begin{equation*}
		\|S_\mathcal{G}(t)\varphi\|_{L^r(\R^{N+k})} 
		\leq \|K(\cdot,x_0,\cdot,t)\|_{L^q(\R^{N+k})}\|\varphi\|_{L^p(\R^{N+k})} 
		\leq C t^{-\frac{N+2k}{2}\left(\frac{1}{p}-\frac{1}{r}\right)} \|\varphi\|_{L^p(\R^{N+k})}.
	\end{equation*}	
	This is \eqref{semigroup}. If $r=p$, with help of Lemma \ref{approxid}, we easily get
	\begin{equation*}
		\|S_\mathcal{G}(t)\varphi-S_\mathcal{G}(t_0)\varphi\|_{L^p} =\|S_\mathcal{G}(t_0)[S_\mathcal{G}(t-t_0)\varphi-\varphi]\|_{L^p} \leq\|S_\mathcal{G}(t-t_0)\varphi-\varphi\|_{L^p}\rightarrow 0,
	\end{equation*}
	as $t\rightarrow t_0\geq0$. If $r>p$, 
	\begin{equation*}
		\|S_\mathcal{G}(t)\varphi-S_\mathcal{G}(t_0)\varphi\|_{L^r} \leq Ct_0^{-\frac{N+2k}{2}\left(\frac{1}{p}-\frac{1}{r}\right)}\|S_\mathcal{G}(t-t_0)\varphi-\varphi\|_{L^p}\rightarrow 0,
	\end{equation*}
	as $t\rightarrow t_0>0$.
\end{proof}


\section{Local solutions}

In this section, we obtain result of existence, uniqueness, continuous dependence and a blow-up alternative for 
\begin{equation}
	\left\{\begin{array}{cc}\label{*}
		u_t-\Delta_\mathcal{G}u=|u|^{\rho-1}u, \ x \in \R^{N+k}, \ t\in (0,T], \\
		\hspace*{-2cm}u(0,x)=u_0(x), \ x \in \R^{N+k},
	\end{array}\right.
\end{equation}em que $u_0 \in L^p(\R^{N+k})$ e $\rho>1$.
The results are similar to those given in \cite{BrezisCaz} and \cite{deAn-Si-Vi-22}, in the case of the classical diffusion equation and its fractional version, respectively.

Let $E=C([0,T];L^p(\R^{N+k}))$ and $X_\alpha$ given by
\begin{equation*}
	X_\alpha=\{u \in C((0,T], E):\|u\|_{X_{\alpha}}=\sup_{t \in (0,T]}t^\alpha\|u(t)\|_{L^r}<+\infty\},
\end{equation*}
with $r=\rho p>p$ and 
\begin{center}
	\fbox{ $\alpha=\frac{N+2k}{2\rho p}\left(\rho-1\right) $.}	
\end{center}

\begin{definition}\label{nonlsemi}
	A continuous function $u:[0,T]\rightarrow L^{p}(\R^{N+k})$ that satisfies
	\begin{equation}\label{branda}
		u(t)=S_{\mathcal{G}}(t)u_0 + \int_{0}^{t}S_{\mathcal{G}}(t-s) |u(s)|^{\rho-1}u(s)ds,
	\end{equation}
	is called a $L^p$-mild solution of $\ref{*}$.
\end{definition}

\begin{remark}
	The semigroup properties of $S_\mathcal{G}(t)$ can be used to prove that a $L^p$-mild solution of $\ref{*}$ defines a nonlinear semigroup $\mathcal{T}_{u_0}(t), t\geq0$, in the sense that $\mathcal{T}_{u_0}(t+s) = \mathcal{T}_{u(t)}u(s)$.
\end{remark}

\begin{theorem}\label{teoexun}
	Let $u_0 \in L^p(\R^{N+k})$, $p>\frac{N+2k}{2}(\rho-1)$, and $\rho>1$. There exists $T>0$ and a unique $L^p$-mild solution $u$ of \eqref{*} in $X_\alpha\cap E$. Moreover, if $u(t)$ and $v(t)$ are mild solutions of $\eqref{*}$, with initial data $u_0$ e $v_0$, respectively, then
	\begin{equation}
		\|u(t)-v(t)\|_{L^p}\leq C\|u_0-v_0\|_{L^p},
	\end{equation}
	for all $t \in [0,T]$.
\end{theorem}

\begin{proof}
	Let $ X=X_\alpha\cap E$ be endowed with the norm
	\begin{equation*}
		\|\cdot\|_X=\|\cdot\|_{X_\alpha}+\|\cdot\|_E.
	\end{equation*}
	Set $M\geq \|u_0\|_{L^p}$ and denote by $S$ the closed ball
	\begin{equation*}
		S=\{u \in X; \|u(t)\|_{L^p}\leq MC+1 \ \text{e} \ t^\alpha\|u(t)\|_{L^{\rho p}}\leq MC+1, \ \text{para} \ t \in (0,T]\}.
	\end{equation*}
	Here $C>0$ is the greater between $1$ and the constant $C(p,\rho p,N)$ in Theorem \ref{kernelpropertys}. Define $\Lambda:S\rightarrow S$ by
	\begin{equation}\label{solop}
		\Lambda(u)(t)=S_{\mathcal{G}}(t)u_0 + \int_{0}^{t}S_{\mathcal{G}}(t-s) |u(s)|^{\rho-1}u(s) ds.
	\end{equation}
	Let us prove that $\Lambda$ is well-defined. For $u \in S$, we have from Theorem \ref{kernelpropertys} that
	\begin{align*}
		\|\Lambda(u)\|_{L^{\rho p}}& \leq \|S_{\mathcal{G}}(t)u_0\|_{L^{\rho p}}+\left\|\int_{0}^{t}S_{\mathcal{G}}(t-s) |u(s)|^{\rho-1}u(s) ds\right\|_{L^{\rho p}} \\[.1cm]
		&\leq C\|u_0\|_{L^p}t^{-\alpha} +  C\left(\sup_{t \in (0,T]}t^\alpha\|u(t)\|_{L^{\rho p}}\right)^\rho\int_{0}^{t}(t-s)^{-\alpha}s^{-\alpha\rho}ds \\[.1cm]
		&= C\|u_0\|_{L^p}t^{-\alpha} + C\|u\|_{X_{\alpha}}^{\rho}t^{-\alpha-\alpha\rho+1}\int_{0}^{1}(1-z)^{-\alpha}z^{-\alpha\rho}dz .
	\end{align*}
	Assumption $p>\frac{(N+2k)(\rho-1)}{2}$ implies both $1-\alpha\rho>0$ and $\frac{(N+2k)(1-\rho)}{2\rho p}+1>0$. Then,
	\begin{align*}
		\|\Lambda(u)\|_{X_\alpha}&\leq C\|u_0\|_{L^p}+\tilde{C}\|u\|_{X_{\alpha}}^{\rho}T^{1-\alpha\rho}.
	\end{align*}
	Hence, we can take $T>0$ sufficiently small so that $\|\Lambda(u)\|_{X_\alpha}\leq MC+1$.

	The continuity of $t \mapsto \Lambda u(\cdot,t)$ in $L^{\rho p}(\R^{N+k})$, for $t>0$, follows from \eqref{continuidadeforte} and its boundedness in $X_\alpha$. Its continuity at $t=0$ in $L^{p}(\R^{N+k})$ relies in the following bound:
	\begin{align*}
		\|\Lambda(u)(t)-u_0\|_{L^{p}} & =  \left\|\int_{0}^{t}S_{\mathcal{G}}(t-s) |u(s)|^{\rho-1}u(s) ds\right\|_{L^{p}} \\
		&\leq C \int_{t_0}^{t}\left\|u(s)\right\|_{L^{\rho p}}^{\rho}ds \\[.1cm]
		&\leq C (M+1)^{\rho} t^{1-\alpha\rho} \rightarrow 0,
	\end{align*}
	as $t\rightarrow 0^+$. The same computation above can be used to prove that
	$$ \|\Lambda(u)(t)\|_{L^{p}} \leq MC + C(M+1)^\rho T^{1-\alpha \rho} .$$
	Thus, for a sufficiently small $T>0$, $\Lambda$ maps $S$ into itself. Let us prove that $\Lambda$ is a contraction. If $u,v \in S$, then
	\begin{align*}
		t^\alpha\|\Lambda(u)-\Lambda(v)\|_{L^{\rho p}} 
		\leq & C t^{\alpha}\int_{0}^{t}(t-s)^{-\alpha}\|u-v\|_{L^{\rho p}}  \left(\|u(s)\|_{L^{\rho p}}^{\rho-1}+\|v(s)\|_{L^{\rho p}}^{\rho-1}\right) ds \\[.1cm]
		\leq & C t^{\alpha}\left[\left(\sup_{t \in (0,T]}t^\alpha\|u(s)\|_{L^{\rho p}}\right)^{\rho-1}+\left(\sup_{t \in (0,T]}t^\alpha\|v(s)\|_{L^{\rho p}}\right)^{\rho-1}\right] \\[.1cm]
		& \ \ \ \ \ \ \ \ \ \ \ \times \int_{0}^{t}(t-s)^{-\alpha}s^{-\alpha(\rho-1)}\|u-v\|_{L^{\rho p}}ds,
	\end{align*}
	that is,
	\begin{equation}\label{resultcontration1}
		\|\Lambda(u)-\Lambda(v)\|_{X_\alpha}\leq C'T^{1-\alpha\rho}\|u-v\|_{X}.
	\end{equation}
	Analogously, we obtain
	\begin{equation}\label{resultcontration2}
		\|\Lambda(u)-\Lambda(v)\|_E\leq C''T^{1-\alpha\rho}\|u-v\|_{X}.
	\end{equation}
	
	From \eqref{resultcontration1} and \eqref{resultcontration2}, for a sufficiently small $T>0$, we have
	\begin{equation*}
		\|\Lambda(u)-\Lambda(v)\|_X\leq k\|u-v\|_{X}, k\in (0,1),
	\end{equation*}
	By the contraction principle, there exists a fixed point for $\Lambda$ that is the unique $L^p$-mild solution of \eqref{*} in $S$. 
	
	To extend this uniqueness to $X$, we take $u,v$ in $X$ and perform as in the part which we proved that $\Lambda$ is a contraction, we get
	\begin{align*}
		\|u(t)-v(t)\|_{L^{\rho p}}&\leq\left(\|u\|_{X\alpha}^{\rho-1}+\|v\|_{X_\alpha}^{\rho-1}\right) \int_{0}^{t}(t-s)^{-\alpha}s^{-\alpha(\rho-1)}\|u(s)-v(s)\|_{L^{\rho p}}ds \\[.1cm]
		&\leq M\int_{0}^{t}(t-s)^{-\alpha}s^{-\alpha(\rho-1)}\|u(s)-v(s)\|_{L^{\rho p}}ds.
	\end{align*}
	
	To use the singular Gronwall inequality (see e.g. \cite{BrezisCaz}), we shall verify that, for $A=a=0,b=\alpha,f(s)=s^{-\alpha(\rho-1)}$, 
	\begin{itemize}
		\item[i)] $\max\{a,b\}=b=\alpha$;
		\item[ii)] $f$ is nonnegative and, for $\tilde{p}>1$, $f \in L^{\tilde{p}}(0,T)$. Equivalently, $-\alpha(\rho-1) \tilde{p}>-1$ or  $\tilde{p}<1/\alpha(\rho-1)$.
		\item[iii)] The conjugated $\tilde{q}$ of $\tilde{p}$ satisfies $\tilde{q}\max\{a,b\}$<1. Equivalently, $\tilde{p}>\dfrac{1}{1-\alpha}$.
	\end{itemize}	
	
	Then, it is enough to notice that $(\frac{1}{1-\alpha},\frac{1}{\alpha(\rho-1)})$ is not degenerate because
	\begin{equation*}
		(1-\alpha)-\alpha(\rho-1)=1-\alpha-\alpha\rho+\alpha=1-\alpha\rho>0.
	\end{equation*}
	
	From the singular Gronwall inequality, $u\equiv v$ in $X_\alpha$. Also, for $u,v\in E$ mild solutions of \eqref{cauchy}-\eqref{cauchy0}, we have
	\begin{equation*}
		\|u(t)-v(t)\|_{L^p}  \leq 2C(M+1)^{\rho-1}\int_{0}^{t}(t-s)^{-\alpha \rho}\|u(s)-v(s)\|_{L^{ p}}ds.
	\end{equation*}
	Thanks to $p>\frac{N+2k}{2}(\rho-1)$, we have $\alpha\rho<1$, and Grownall's inequality gives that $u\equiv v$ also in $E$. Therefore, the uniqueness in $X$ is proved.

	For the continuous dependence, let $u,v\in X$ given by
	\begin{align}
		u(t)&=S_\mathcal{G}(t)u_0+\int_{0}^{t}S_\mathcal{G}(t)|u(s)|^{\rho-1}u(s)ds,\label{dpu} \\[.1cm]
		v(t)&=S_\mathcal{G}(t)v_0+\int_{0}^{t}S_\mathcal{G}(t)|v(s)|^{\rho-1}v(s)ds.\label{dpv}
	\end{align}
	Performing as above, we get
	\begin{equation}\label{deplin1}
		\|u(t)-v(t)\|_{L^p}\leq C\|u_0-v_0\|_{L^p}+C_3'\sup_{t \in (0,T]}t^\alpha\|u(t)-v(t)\|_{L^{\rho p}}.
	\end{equation}
	On the other hand,
	\begin{align*}
		\|u(t)-v(t)\|_{L^{\rho p}}\leq C_4\|u_0-v_0\|_{L^p}t^{-\alpha}+C_5A\int_{0}^{t}(t-s)^{-\alpha}s^{-\alpha\rho}s^{\alpha}\|u-v\|_{L^{\rho p}}ds.
	\end{align*}
	Singular Grownall inequality implies 
	\begin{equation}\label{deplin2}
		\|u(t)-v(t)\|_{L^{\rho p}}\leq C_7\|u_0-v_0\|_{L^p}t^{-\alpha}.
	\end{equation}
	Accordingly, we combine \eqref{deplin1} and \eqref{deplin2} to obtain
	\begin{equation}
		\|u(t)-v(t)\|_{L^p}\leq L\|u_0-v_0\|_{L^{p}}.
	\end{equation}
\end{proof}

We want to extend the solutions found in Theorem \ref{teoexun}. Thus, we say that a $L^p$-mild solution $u:[0,T]\to L^p(\R^{N+k})$ of \eqref{*}, we say that $\bar{u}:[0,\bar{T}]\to L^p(\R^{N+k})$ is a \textit{continuation} of $u$ on $[0,\bar{T}]$ if $\bar{u}$ is a mild solution, with $\bar{T}>T$ and $u(t)=\bar{u}(t)$ whenever $t\in[0,T]$. We say that an extension $\bar{u}$ defined on $[0,T_m)$ is a \textit{maximal solution} if it cannot be extended.

\begin{theorem}
	Let $u_0 \in L^p(\R^{N+k})$, $p>(N+1)(\rho-1)/2$, $\rho>1$. Then, there exists, a unique continuation of $u$ of \eqref{*}, the solution guaranteed by Theorem \ref{teoexun}. Moreover, there exists a maximal solution such that, if $T_m<\infty$ then
	\begin{equation}\label{limblowup}
		\lim\limits_{t\rightarrow T_m^-}\|u(t)\|_{L^{\rho p}}=+\infty.
	\end{equation}
\end{theorem}
\begin{proof}
	Let $u$ be the solution guaranteed by Theorem \ref{teoexun} and defined on $[0,T_0]$. Consider the problem
	\begin{equation}\label{auxcont}
		\left\{\begin{array}{cc}
			v_t-\Delta_{\mathcal{G}}v=|v|^{\rho-1}v, \\
			v(0)=u(T_0).
		\end{array}\right.
	\end{equation}
	
	From Theorem \ref{teoexun} again, there is a unique solution $u_1$ of \eqref{auxcont} in $X$ and defined on $[0,T_1]$. Define $\tilde{u}$ by
	\begin{equation}\label{colagemsolu}
		\tilde{u}(t)=\left\{\begin{array}{cc}
			u(t), & \ \text{if} \ t \in [0,T_0], \\
			u_1(t-T_0), &  \ \text{if} \ t \in [T_0,T_1], T_1=T_0+\delta_1.
		\end{array}\right.
	\end{equation}
	It is clear that $\tilde{u}$ continuous and it is not difficult to prove that $\tilde{u}$ is a mild solution for \eqref{*}. Indeed, for $t \in [T_0,T_1]$, $t-T_0 \in [0, \delta_1]$ and
	\begin{align*}
		\tilde{u}(t) 
		&=S_\mathcal{G}(t-T_0)u(T_0)+\int_{0}^{t-T_0}S_\mathcal{G}(t-s)|u_1(s)|^{\rho-1}u_1(s)ds \\[.1cm]
		&=S_\mathcal{G}(t)u_0+\int_{0}^{T_0}S_\mathcal{G}(t-s)|u(s)|^{\rho-1}u(s)ds \\[.1cm]
		& \ \ \ +\int_{0}^{t-T_0}S_\mathcal{G}(t-T_0-s)|u_1(s)|^{\rho-1}u_1(s)ds \\
		&=S_\mathcal{G}(t)u_0+\int_{0}^{t}S_\mathcal{G}(t-s)|\tilde{u}(s)|^{\rho-1}\tilde{u}(s)ds.
	\end{align*}
	This proves the claim. Repeating the reasoning, we have $T_0<T_1<\cdots <T_j< \cdots$. By Zorn's Lemma, $A$ has a maximal element $T_m$, from which follows the existence of the maximal solution of $\eqref{*}$.
	
	Next, assume $T_m<\infty$ but \eqref{limblowup} does not hold, that is, there exists a constant $c>0$ such that $\|u(t)\|_{L^{\rho p}}\leq c<\infty$ for all $t \in [0,T_m)$. Given a sequence, $(t_n)_{n\in \N}\subset \R_+$, $t_n\rightarrow T_m^-$, we take $(u(t_n))_{n \in \N}$. We will show that this is a Cauchy sequence. For $0<t_m<t_n<T_m$, we can use the properties of the semigroup $S_{\mathcal{G}}(t)$ in Theorem :
	\begin{align*}
		u(t_n)-u(t_m)& = S_\mathcal{G}(t_m) [S_\mathcal{G}(t_n-t_m)-I]u_0+\int_{t_m}^{t_n}S_\mathcal{G}(t_n-s)|u(s)|^{\rho-1}u(s)ds \\[.1cm]
		& \ \ \ + [S_\mathcal{G}(t_n-t_m) - I] \int_{0}^{t_m} S_\mathcal{G}(t_m-s)]|u(s)|^{\rho-1}u(s)ds  .
	\end{align*}

	Since $u_0\in L^p(\R^{N+k}) \cap L^{\rho p}(\R^{N+k})$ for $t>0$ and we proved in Theorem \ref{teoexun} that $t\mapsto \int_{0}^{t_m} S_\mathcal{G}(t_m-s)]|u(s)|^{\rho-1}u(s)ds $ also is in $ L^p(\R^{N+k}) \cap L^{\rho p}(\R^{N+k})$, we can use the strong continuity of the semigroup $S_\mathcal{G}(t)$ given in \eqref{continuidadeforte} to conclude that the first and third terms of the right-hand side above go to zero as $m,n\rightarrow \infty$.
	
	Plus, notice that
	\begin{align*}
		\left\|\int_{t_m}^{t_n}S_\mathcal{G}(t_n-s)|u(s)|^{\rho-1}u(s)ds\right\|_{L^{\rho p}} 
		&\leq\int_{t_m}^{t_n}C(t_n-s)^{-\alpha}\|u(s)\|^\rho_{L^{\rho p}}ds \\[.1cm]
		&\leq CA^\rho  t_n^{-\alpha-\alpha\rho+1}\int_{t_m/t_n}^{1}(1-z)^{-\alpha}z^{-\alpha\rho}dz \rightarrow 0,
	\end{align*}
	as $n,m \rightarrow \infty$. Then,
	\begin{equation*}
		\|u(t_n,\cdot)-u(t_m,\cdot)\|_{L^{\rho p}}\rightarrow 0, \ \text{as}\ n,m \rightarrow \infty.
	\end{equation*}
	Analogously,
	\begin{align*}
		\left\|\int_{t_m}^{t_n}S_\mathcal{G}(t_n-s)|u(s)|^{\rho-1}u(s)ds\right\|_{L^{p}}
		&\leq C(M+1)^\rho\int_{t_m}^{t_n}s^{-\alpha\rho}ds \\[.1cm]
		&\leq C(M+1)^\rho(t_n^{1-\alpha\rho}-t_m^{1-\alpha\rho})\rightarrow 0,
	\end{align*}
	as $n,m \rightarrow \infty$, yielding
	\begin{equation*}
		\|u(t_n,\cdot)-u(t_m,\cdot)\|_{L^{p}}\rightarrow 0, \ \text{as}\ n,m \rightarrow \infty.
	\end{equation*}
	Accordingly,
	\begin{equation*}
		\|u(t_n,\cdot)-u(t_m,\cdot)\|_{X}\rightarrow 0, \ \text{as}\ n,m \rightarrow \infty.
	\end{equation*}
	It shows that $(u(t_n))_{n\in\N}$ is a Cauchy sequence in the Banach space $X$, which implies the existence of the limit in $X$
	\begin{equation*}
		\lim_{n\rightarrow \infty}u(t_n)=u(T_m).
	\end{equation*}
	It contradicts the maximality of $T_m$.
\end{proof}

\section{Global solutions}

In this section, we assume $p=\frac{N+2k}{2}(\rho-1)$ to prove global existence for $\eqref{*}$. In this case, $r \in (p,\rho p)$ and
\begin{center}
	\begin{equation}\label{alpha2}
		\boxed{		\alpha=\frac{N+2k}{2}\left(\frac{1}{p}-\frac{1}{r}\right). 	}
	\end{equation}
\end{center}

We need the following lemma.
\begin{lemma}\label{lemmaSg}
	Let $1\leq p<r<\infty$. Then
	\begin{equation}
		\lim\limits_{T\rightarrow 0^+}\sup_{t \in (0,T)}t^{\alpha}\|S_\mathcal{G}(t)\varphi\|_{L^r}=0
	\end{equation}
	uniformly to $\varphi$ in bounded sets of $L^p(\R^{N+k})$.
\end{lemma}

\begin{proof}
	Given $\varphi \in L^p(\R^{N+k})$, there exists a sequence $(\varphi_n)_{n \in \N}\subset C_0^\infty(\R^{N+k})$ converging to $\varphi$. Using \eqref{semigroup}, we have
	\begin{align*}
		t^\alpha\|S_\mathcal{G}(t)\varphi\|_{L^r}
		&\leq C\|\varphi-\varphi_n\|_{L^r}+t^\alpha M, \ \ \forall n\in \N, t\in(0,T).
	\end{align*}
	
	Taking the limit as $n \rightarrow \infty$ and the supremum, we get
	\begin{equation*}
		\lim\limits_{T\rightarrow 0^+}\sup_{t \in (0,T)}t^{\alpha}\|S_\mathcal{G}(t)\varphi\|_{L^r}=0.
	\end{equation*}
\end{proof}

For the next result, in the spaces $E$ and $X_\alpha$, we will consider $T=\infty$.
\begin{theorem}
	Let $p=\frac{N+2k}{2}(\rho-1)$ and $u_0\in L^p(\R^{N+k})$, with $\|u_0\|_{L^p}$ sufficiently small. Then, there exists a global $L^p$-mild solution $u\in X_\alpha\cap E$.
\end{theorem}

\begin{proof}
	Let $\mathcal{B}_\delta$ be the closed ball of $X_\alpha\cap E$ of radius $\delta>0$ and  $\Lambda:\mathcal{B}_\delta \to \mathcal{B}_\delta$ given by \eqref{solop}. Analogously to the proof of Theorem \ref{teoexun}, for any $u\in \mathcal{B}_\delta$, we can use Theorem \ref{teo3.6} to get
	\begin{align*}
		\|\Lambda(u)(t)\|_{L^p} 	& \leq \|u_0\|_{L^p}+C\left(\sup_{t \in (0,T)}t^\alpha\|u(t)\|_{L^r}\right)^{\rho} \int_{0}^{t} (t-s)^{\frac{N+2k}{2}\left(\frac{\rho}{r} - \frac{1}{p}\right)} s^{-\alpha\rho}ds.
	\end{align*}
	The assumptions $r\in(p,\rho p)$ and $p=\frac{N+2k}{2}(\rho-1)$ make $\frac{N+2k}{2}\left(\frac{\rho}{r}  - \frac{1}{p}\right) > - 1$ and, because of \eqref{alpha2}, we have
	$$1 -\frac{N+2k}{2}\left(\frac{\rho}{r} - \frac{1}{p}\right) - \alpha\rho = 0 .$$
	Then
	\begin{equation*}
		\|\Lambda(u)(t)\|_{L^p} \leq \|u_0\|_{L^p} + \tilde{C}\delta^\rho.
	\end{equation*}
	
	Similarly,
	\begin{align*}
		\|\Lambda(u)(t)\|_{L^r} 
		&\leq \|u_0\|_{L^p}t^{-\alpha}+C\delta^{\rho}t^{-\frac{N+2k}{2r}(\rho-1)-\alpha\rho+1}\int_{0}^{1}(1-\tau)^{-\frac{N+2k}{2r}(\rho-1)}\tau^{-\alpha\rho}d\tau . 
	\end{align*}
	Again, we notice that $r\in(p,\rho p)$ and the choice for $p$ make $\frac{N+2k}{2r}(\rho-1)<1$, $\alpha \rho<1$, and $\alpha=\frac{N+2k}{2r}(\rho-1)+\alpha\rho-1$. Thus,
	\begin{equation*}
		t^\alpha\|\Lambda(u)(t)\|_{L^r} \leq \|u_0\|_{L^p} + \tilde{C_1}\delta^\rho
	\end{equation*}
	
	If $ \|u_0\|_{L^p}<\delta/2$ and $\delta>0$ is so small that $4(\tilde{C}+\tilde{C_1})\delta^{\rho-1}<1$, then
	\begin{equation}
		\|\Lambda u \|_X< \frac{3\delta}{4} .
	\end{equation}
	It is similar to the proof of Theorem \ref{teoexun} to prove that $t\mapsto \Lambda u (t)$ is in $E\cap X_\alpha$, however Lemma \ref{lemmaSg} must be invoked. Therefore, $\Lambda$ is well-defined.

	Moreover, for $u,v \in \mathcal{B}_\delta$, we have
	\begin{align*}
		& \|\Lambda(u)(t)-\Lambda(v)(t)\|_{L^r} \\
		&\leq C\int_{0}^{t}(t-s)^{-\frac{(N+2k)(\rho-1)}{2r}}\|u(s)-v(s)\|_{L^r}\left(\|u(s)\|_{L^r}^{\rho-1}+\|v(s)\|_{L^r}^{\rho-1}\right)ds \\[.1cm]
		& \leq \sup_{t >0}t^\alpha\|u(t)-v(t)\|_{L^r} \int_{0}^{t}(t-s)^{-\frac{(N+2k)(\rho-1)}{2r}}s^{-\alpha(\rho-1)-\alpha}ds \\[.1cm]
		&\leq 2C\delta^{\rho-1}t^{-\frac{(N+2k)(\rho-1)}{2r}-\alpha\rho+1}\|u-v\|_{X_\alpha}B\left(1-\frac{(N+2k)(\rho-1)}{2r},1-\alpha\rho\right),
	\end{align*}
	that is,
	\begin{equation}\label{rhoigual1}
		\|\Lambda(u)-\Lambda(v)\|_{X_\alpha}\leq 2C\delta^{\rho-1}\|u-v\|_{X}.
	\end{equation}
	One can proceed similarly to show
	\begin{equation}\label{rhoigual2}
		\|\Lambda(u)(t)-\Lambda(v)(t)\|_{L^p}\leq C'\delta^{\rho-1}\|u-v\|_{X}, t>0.
	\end{equation}
	Hence, from \eqref{rhoigual1} and \eqref{rhoigual2}, $\Lambda$ is a contraction. The Banach fixed point theorem yields a unique mild solution for \eqref{sheatG}--\eqref{sheatG0} in $\mathcal{B}_\delta \subset E\cap X_\alpha$.

\end{proof}

\begin{remark}
	\begin{enumerate}
		\item If the smallness of the norm of the initial datum is replaced with the smallness of the time, we could prove the existence of a local solution \eqref{sheatG}--\eqref{sheatG0}, with $p=\frac{N+2k}{2}(\rho-1)$ and arbitrary size of the initial datum.
		\item 	Let $u,v$ be solutions in $X$ starting at $u_0,v_0$, respectively. In the same manner we obtained \eqref{rhoigual2}, we can show that
		\begin{equation*}
			\|u(t)-v(t)\|_{X_\alpha}+\|u(t)-v(t)\|_{L^p}\leq \tilde{L}\|u_0-v_0\|_{L^p}.
		\end{equation*}
		\item From the positivity of the heat kernel \eqref{HK}, on can obtain that the semigroup $(S_\mathcal{G}(t))_{t\geq0}$ preserves positivity. Then, we can follow the same argument in \cite[Th. 5.1]{deAn-Si-Vi-22} to infer that the solution we find in Theorems \ref{teoexun} and \ref{teo3.6} are positive, provided the initial data is nonnegative.
	\end{enumerate}
\end{remark}


\end{document}